\documentclass[12pt]{article}

\usepackage{latexsym,amsmath,amscd,amssymb,graphics,float}
\usepackage{enumerate}

\usepackage{graphicx,lscape}

\usepackage[colorlinks]{hyperref}
\usepackage{url}

\begin{document}

\newenvironment{proof}[1][Proof]{\textbf{#1.} }{\ \rule{0.5em}{0.5em}}

\newtheorem{theorem}{Theorem}[section]
\newtheorem{definition}[theorem]{Definition}
\newtheorem{lemma}[theorem]{Lemma}
\newtheorem{remark}[theorem]{Remark}
\newtheorem{proposition}[theorem]{Proposition}
\newtheorem{corollary}[theorem]{Corollary}
\newtheorem{example}[theorem]{Example}

\numberwithin{equation}{section}
\newcommand{\ep}{\varepsilon}
\newcommand{\R}{{\mathbb  R}}
\newcommand\C{{\mathbb  C}}
\newcommand\Q{{\mathbb Q}}
\newcommand\Z{{\mathbb Z}}
\newcommand{\N}{{\mathbb N}}

\newcommand{\bfi}{\bfseries\itshape}

\newsavebox{\savepar}
\newenvironment{boxit}{\begin{lrbox}{\savepar}
\begin{minipage}[b]{15.5cm}}{\end{minipage}\end{lrbox}
\fbox{\usebox{\savepar}}}

\title{{\bf A note on the structure of prescribed gradient--like domains of non--integrable vector fields}}
\author{R\u{a}zvan M. Tudoran}

\date{}
\maketitle \makeatother

\begin{abstract}
Given a geometric structure on $\mathbb{R}^{n}$ with $n$ even (e.g. Euclidean, symplectic, Minkowski, pseudo-Euclidean), we analyze the set of points inside the domain of definition of an arbitrary given $\mathcal{C}^1$ vector field, where the value of the vector field equals the value of the left/right gradient--like vector field of some fixed $\mathcal{C}^2$ potential function, although a non-integrability condition holds at each such a point. Particular examples of gradient--like vector fields include the class of gradient vector fields with respect to Euclidean or pseudo-Euclidean inner products, and the class of Hamiltonian vector fields associated to symplectic structures on $\mathbb{R}^{n}$ (with $n$ even). The main result of this article provides a geometric version of the main result of \cite{della}. 
\end{abstract}

\medskip



\section{Introduction}
\label{section:one}

One of the most important classes of vector fields on $\mathbb{R}^{n}$ is the class of \textit{gradient} vector fields. In order to be precise, recall that the notion of gradient vector field on $\mathbb{R}^{n}$ is associated to some Euclidean or pseudo-Euclidean inner product on $\mathbb{R}^{n}$ (or more generally, to some Riemannian or pseudo-Riemannain metric). More exactly, given an Euclidean or pseudo-Euclidean inner product $g:\mathbb{R}^{n}\times \mathbb{R}^{n}\rightarrow \mathbb{R}$ and a function $f\in\mathcal{C}^{1}(\Omega,\mathbb{R})$ (defined on some open set $\Omega\subseteq\mathbb{R}^{n}$), the gradient vector field $\nabla_{g}f\in\mathfrak{X}(\Omega)$ is uniquely defined by the relation
\begin{equation}\label{grf}
g(\nabla_{g}f(\mathbf{x}),\mathbf{v})=\mathrm{d}f(\mathbf{x})\cdot\mathbf{v}, ~ \forall \mathbf{x}\in \Omega,\forall\mathbf{v}\in T_{\mathbf{x}}\Omega\cong \mathbb{R}^{n}.
\end{equation}
A natural question regarding these vector fields is when a given $\mathcal{C}^{1}$ vector field $F\in\mathfrak{X}(\Omega)$ is the gradient of some potential function $f\in\mathcal{C}^{2}(\Omega,\mathbb{R})$, with respect to a fixed Euclidean or pseudo-Euclidean inner product? The answer to this question in the case when $\Omega$ is a contractible open set, is given by the Poincar\'e lemma which provides an integrability condition to be verified by $F$ at each point of $\Omega$. 

In this note we are concerned with a geometric extension of the main result of \cite{della} which roughly states that given a $\mathcal{C}^{1}$ vector field $F\in\mathfrak{X}(\Omega)$ defined on an open set $\Omega\subseteq\mathbb{R}^{2m}$, and a function $f\in\mathcal{C}^{2}(\Omega,\mathbb{R})$, the set of points $\mathbf{x}\in \Omega$ where $F(\mathbf{x})=\nabla f(\mathbf{x})$ although a natural non-integrability condition holds at each such a point, if not empty, is a Borel set covered by a finite family of regularly embedded $\mathcal{C}^{1}$ $m-$dimensional submanifolds of $\mathbb{R}^{2m}$; the symbol $\nabla$ stands for the gradient operator associated to the canonical Euclidean inner product on $\mathbb{R}^{2m}$. More precisely, we extend the above result in the context of gradient--like vector fields associated to a general geometric structure on $\mathbb{R}^{2m}$, where particular cases of gradient--like vector fields include the gradient vector fields associated to Euclidean or pseudo-Euclidean inner products, and also the Hamiltonian vector fields associated to symplectic forms on $\mathbb{R}^{2m}$. 

The structure of the article is the following. In the second section we recall some generalities about gradient--like vector fields associated to a geometric structure on $\mathbb{R}^{n}$, and provide an equivalent characterization of the integrability condition from the Poincar\'e lemma in the case of gradient--like vector fields. The third section contains the statement and the proof of the main result of this note, together with some particular cases.

\section{Gradient--like vector fields associated to a geometric structure}

We start this section by recalling from \cite{TDR} the definition of a \textit{geometric structure} on $\mathbb{R}^{n}$. More precisely, following \cite{TDR}, by \textit{geometric structure} on $\mathbb{R}^{n}$ we mean any arbitrary fixed nondegenerate real bilinear form on $\mathbb{R}^n$. From now on, the pair $(\mathbb{R}^{n},b)$, where $b$ is a geometric structure on $\mathbb{R}^n$, will be denoted by $\mathbb{R}^{n}_{b}$.

Note that Euclidean, symplectic, Minkowski, or pseudo-Euclidean structures, are all particular examples of geometric structures. Indeed, if $b$ is symmetric and positive definite, then $\mathbb{R}^{n}_{b}$ becomes an Euclidean vector space; if $n$ is even and $b$ is skew--symmetric, then $\mathbb{R}^{n}_{b}$ becomes a symplectic vector space; if $b$ is symmetric with the signature $(n-1,1)$, then $\mathbb{R}^{n}_{b}$ becomes a Minkowski vector space, or more generally, if $b$ is symmetric with the signature $(n-k,k)$, $k\neq 0$, then $\mathbb{R}^{n}_{b}$ becomes a pseudo-Euclidean vector space.

Let us recall now a result from \cite{TDR} which is a consequence of the nondegeneracy of geometric structures, and provides a compatibility relation between a general geometric structure on $\mathbb{R}^{n}$, and the canonical Euclidean inner product on $\mathbb{R}^{n}$. More precisely, for any arbitrary fixed geometric structure $b$ on $\mathbb{R}^{n}$, there exists a \textit{unique} invertible linear map $B\in\operatorname{Aut}(\mathbb{R}_{b}^{n})$ such that 
\begin{equation}\label{relb}
\langle \mathbf{x},\mathbf{y} \rangle = b(\mathbf{x}, B\mathbf{y}), ~~\forall \mathbf{x},\mathbf{y}\in\mathbb{R}^{n}_{b},
\end{equation}
where $\langle \cdot,\cdot \rangle$ stands for the canonical Euclidean inner product on $\mathbb{R}^{n}$. Following \cite{TDR}, a pair $(b,B)$ satisfying the relation \eqref{relb} is called a \textit{geometric pair} on $\mathbb{R}^{n}$.

Next, given a geometric pair $(b,B)$, the \textit{(left) adjoint} of $B$ with respect to the geometric structure $b$ is the unique invertible linear map, denoted by $B^{\star}\in\operatorname{Aut}(\mathbb{R}_{b}^{n})$, defined by the relation
\begin{equation*}
b(B^{\star}\mathbf{x}, \mathbf{y})= b(\mathbf{x}, B\mathbf{y}), ~~\forall \mathbf{x},\mathbf{y}\in\mathbb{R}^{n}_{b}.
\end{equation*}

Recall from \cite{TDR} that any geometric pair $(b,B)$ with $b$ symmetric or skew--symmetric, yields $\left(B^{\star}\right)^{\star}= B$. For a general geometric pair $(b,B)$ it could happen that $\left(B^{\star}\right)^{\star}\neq B$. Next result from \cite{TDR} recalls some basic properties of a geometric pair $(b,B)$.

\begin{proposition}\cite{TDR}\label{prp2}
Let $b$ be a geometric structure on $\mathbb{R}^{n}$, and let $(b,B)$ be the associated geometric pair. Then $B^{\star}= B^{\top}$, where $B^{\top}$ stands for the adjoint of $B$ with respect to $\langle\cdot,\cdot\rangle$. Moreover, if $b$ is symmetric (skew--symmetric) then $B^{\star}=B$ ($B^{\star}= - B$).
\end{proposition}

In the following we recall from \cite{TDR} the definition of two gradient--like vector fields naturally associated to a given geometric structure. These vector fields generalize the notion of gradient vector field from the Euclidean geometry, Hamiltonian vector field from the symplectic geometry, and Minkowski gradient vector field from the Minkowski geometry.

\begin{definition}\cite{TDR}
Let $b$ be a geometric structure on $\mathbb{R}^n$, and let $\Omega\subseteq\mathbb{R}^{n}_{b}$ be an open set. Then for every $f\in\mathcal{C}^{1}(\Omega,\mathbb{R})$ we define the left--gradient of $f$, denoted by $\nabla^{L}_{b}f$, as the vector field uniquely defined by the relation
$$
b(\nabla^{L}_{b}f(\mathbf{x}),\mathbf{v})=\mathrm{d}f(\mathbf{x})\cdot \mathbf{v}, ~ \forall \mathbf{x}\in \Omega,\forall\mathbf{v}\in T_{\mathbf{x}}\Omega\cong \mathbb{R}^{n}_{b}.
$$
Similarly, the right--gradient of $f$, denoted by $\nabla^{R}_{b}f$, is the vector field uniquely defined by the relation
$$
b(\mathbf{v},\nabla^{R}_{b}f(\mathbf{x}))=\mathrm{d}f(\mathbf{x})\cdot \mathbf{v}, ~ \forall \mathbf{x}\in \Omega,\forall\mathbf{v}\in T_{\mathbf{x}}\Omega\cong \mathbb{R}^{n}_{b}.
$$
\end{definition}

Recall from \cite{TDR} that if $b$ is symmetric then $\nabla^{L}_{b}=\nabla^{R}_{b}$. Particularly, when $b$ is the canonical Euclidean inner product on $\mathbb{R}^n$, then $\nabla^{L}_{b}=\nabla^{R}_{b}=\nabla$, where $\nabla$ stands for the classical gradient operator on $\mathbb{R}^n$, defined for any function $f\in\mathcal{C}^{1}(\Omega,\mathbb{R})$ by the relation
\begin{equation}\label{cg}
\langle\nabla f(\mathbf{x}),\mathbf{v}\rangle=\mathrm{d}f(\mathbf{x})\cdot \mathbf{v}, ~ \forall \mathbf{x}\in\Omega,\forall\mathbf{v}\in T_{ \mathbf{x}}\Omega\cong \mathbb{R}^{n}_{b}.
\end{equation}

On the other hand, if $b$ is skew--symmetric then $\nabla^{L}_{b}= - \nabla^{R}_{b}$. For a general geometric structure $b$, the relation between the operators $\nabla^{L}_{b}$, $\nabla^{R}_{b}$ is given in the following result.

\begin{proposition}\cite{TDR}\label{rlgrd}
Let $b$ be a geometric structure on $\mathbb{R}^{n}$, and let $(b,B)$ be the associated geometric pair. Then for any fixed open set $\Omega\subseteq\mathbb{R}^{n}_{b}$, and for every $f\in\mathcal{C}^{1}(\Omega,\mathbb{R})$, we have that $\nabla^{L}_{b}f=B^{\star}\nabla f$, $\nabla^{R}_{b}f=B \nabla f$, and $\nabla^{L}_{b}f=B^{\star}B^{-1}\nabla^{R}_{b}f$.
\end{proposition}

In the following we recall from \cite{TDR} the Poincar\'e lemma in the case of left/right--gradient vector fields associated to a general geometric structure on $\mathbb{R}^n$.

\begin{theorem}\label{Poincare}\cite{TDR}
Let $b$ be a geometric structure on $\mathbb{R}^{n}$, and let $(b,B)$ be the associated geometric pair. Let $F\in\mathfrak{X}(\Omega)$ be a $\mathcal{C}^{1}$ vector field defined on a contractible open set $\Omega\subseteq\mathbb{R}^{n}$. Then the following assertions hold true.
\begin{enumerate}
\item[(i)] $F=\nabla^{L}_{b}f$ for some $f\in\mathcal{C}^{2}(\Omega,\mathbb{R})$ if and only if
\begin{equation}\label{relunu}
(\mathrm{D}F(\mathbf{x}))^{\top}B^{-1}=(B^{\star})^{-1}\mathrm{D}F(\mathbf{x}), ~ \forall \mathbf{x}\in\Omega.
\end{equation}
\item[(ii)] $F=\nabla^{R}_{b}f$ for some $f\in\mathcal{C}^{2}(\Omega,\mathbb{R})$ if and only if
\begin{equation}\label{reldoi}
(\mathrm{D}F(\mathbf{x}))^{\top}(B^{\star})^{-1}=B^{-1}\mathrm{D}F(\mathbf{x}), ~ \forall \mathbf{x}\in\Omega.
\end{equation}
\item[(iii)] If $b$ is symmetric (i.e. $b$ is an Euclidean or pseudo-Euclidean inner product), then $F=\nabla^{L}_{b}f$ (or equivalently  $F=\nabla^{R}_{b}f$) for some $f\in\mathcal{C}^{2}(\Omega,\mathbb{R})$ if and only if
\begin{equation}\label{reltrei}
(\mathrm{D}F(\mathbf{x}))^{\top}B^{-1}=B^{-1}\mathrm{D}F(\mathbf{x}), ~ \forall \mathbf{x}\in\Omega.
\end{equation}
Here, $\nabla^{L}_{b}f$ (or equivalently $\nabla^{R}_{b}f$) is precisely the gradient of $f$ with respect to the inner product $b$, i.e. $\nabla^{L}_{b}f = \nabla^{R}_{b}f=\nabla_{b}f$.
\item[(iv)] If $n$ is even and $b$ is skew--symmetric (i.e. $b$ is a symplectic form), then $F=\nabla^{L}_{b}f$ (or equivalently  $-F=\nabla^{R}_{b}f$) for some $f\in\mathcal{C}^{2}(\Omega,\mathbb{R})$ if and only if
\begin{equation}\label{relpatru}
(\mathrm{D}F(\mathbf{x}))^{\top}B^{-1} + B^{-1}\mathrm{D}F(\mathbf{x}) = \mathbb{O}_{n}, ~ \forall \mathbf{x}\in\Omega.
\end{equation}
Here, $\nabla^{L}_{b}f$ (or equivalently $-\nabla^{R}_{b}f$) is precisely the Hamiltonian vector field of $f$ with respect to the symplectic form $b$, i.e. $\nabla^{L}_{b}f = -\nabla^{R}_{b}f= X_f$.
\end{enumerate}
\end{theorem}

Note that using differential forms, the first two items of Theorem \ref{Poincare} follow also from the standard Poincar\'e lemma applied to the $1-$forms $\alpha^{L}:=\iota_{F}b$ and $\alpha^{R}:=\iota_{F}b^{op}$, where $b^{op}(F(\mathbf{x}),\mathbf{v}):=b(\mathbf{v},F(\mathbf{x})), \forall \mathbf{x}\in\Omega, \forall\mathbf{v}\in T_{\mathbf{x}}\Omega\cong \mathbb{R}^{n}_{b} $. Hence, the relation \eqref{relunu} is equivalent to $\mathrm{d}\alpha^{L}=0$, and similarly, the relation \eqref{reldoi} is equivalent to $\mathrm{d}\alpha^{R}=0$.

Next we recall a definition from \cite{della} useful in order to provide another formulation of the integrability conditions \eqref{relunu} and \eqref{reldoi}.

\begin{definition}\cite{della}
Let $\{u_1,\dots, u_n\}\subset\mathbb{R}^{n}$ be an orthonormal basis of $\mathbb{R}^{n}$, and let $\Gamma :\operatorname{End}(\mathbb{R}^{n})\rightarrow \bigwedge^{2}\mathbb{R}^{n}$ be the operator given by
$$
\Gamma(M):=\sum_{i=1}^{n}(Mu_i)\wedge u_i,~ \forall M\in\operatorname{End}(\mathbb{R}^{n}).
$$
For $n=2m\neq 0$ one defines $\Gamma(M)^{m}:=\underbrace{\Gamma(M)\wedge\dots\wedge\Gamma(M)}_{m\  \mathrm{times}},~ \forall M\in\operatorname{End}(\mathbb{R}^{n}).$
\end{definition}

In \cite{della} it is proved that the operator $\Gamma$ is well defined (i.e. its formula is independent of the choice of orthonormal basis), and moreover, $\Gamma(M)=0$ if and only if $M^{\top}=M$, where $M^{\top}$ stands for the adjoint of $M$ with respect to the canonical Euclidean inner product $\langle\cdot,\cdot\rangle$ on $\mathbb{R}^n$. Using this characterization of self--adjoint operators (with respect to $\langle\cdot,\cdot\rangle$) and the Proposition \ref{prp2}, we obtain the following equivalent way to express the integrability conditions \eqref{relunu}, \eqref{reldoi}.

\begin{proposition}
Let $b$ be a geometric structure on $\mathbb{R}^{n}$, and let $(b,B)$ be the associated geometric pair. Let $F\in\mathfrak{X}(\Omega)$ be a $\mathcal{C}^{1}$ vector field defined on a contractible open set $\Omega\subseteq\mathbb{R}^{n}$. Then the following assertions hold true.
\begin{enumerate}
\item[(i)] The integrability condition \eqref{relunu} is equivalent to $\Gamma((B^{\star})^{-1}DF(\mathbf{x}))=0, ~\forall\mathbf{x}\in\Omega$.
\item[(ii)] The integrability condition \eqref{reldoi} is equivalent to $\Gamma(B^{-1}DF(\mathbf{x}))=0, ~\forall\mathbf{x}\in\Omega$.
\end{enumerate}
\end{proposition}

\section{The main result and some particular cases}

In this section we analyze the effect of a natural non-integrability condition applied to a general $\mathcal{C}^{1}$ vector field defined on some open subset of $\mathbb{R}^{n}_{b}$, where $n$ is a nonzero even number, and $b$ is an arbitrary fixed geometric structure. Before doing that, let us point out that if $n=2m\neq 0$, then for any $M\in\operatorname{End}(\mathbb{R}^{n})$ with $\Gamma(M)=0$ we get $\Gamma(M)^{m}=0$. Hence, for any $M\in\operatorname{End}(\mathbb{R}^{n})$ with $\Gamma(M)^{m}\neq 0$ it follows that $\Gamma(M)\neq 0$. Let us state now the main result of this note, which is a geometric version of the main result of \cite{della}.

\begin{theorem}\label{MTHM}
Let $b$ be a geometric structure on $\mathbb{R}^{2m}$, and let $(b,B)$ be the associated geometric pair. Let $f\in\mathcal{C}^{2}(\Omega,\mathbb{R})$ and $F\in\mathcal{C}^{1}(\Omega,\mathbb{R}^{2m}_{b})$, where $\Omega$ is an open subset of $\mathbb{R}^{2m}_{b}$. Then each of the sets
\begin{align*}
A_{f,F}^{\star,L}:&=\{\mathbf{x}\in\Omega\mid \nabla^{L}_{b}f(\mathbf{x})=F(\mathbf{x}),~ \Gamma((B^{\star})^{-1}DF(\mathbf{x}))^{m}\neq 0\},\\
A_{f,F}^{\star,R}:&=\{\mathbf{x}\in\Omega\mid \nabla^{R}_{b}f(\mathbf{x})=F(\mathbf{x}),~ \Gamma(B^{-1}DF(\mathbf{x}))^{m}\neq 0\},
\end{align*} 
if not empty, is a Borel set which admits a cover consisting of at most $\binom{2m}{m}$ regularly embedded $\mathcal{C}^{1}$ $m-$dimensional submanifolds of $\mathbb{R}^{2m}_{b}$. Particularly, the Hausdorff dimension of $A_{f,F}^{\star,L}$ and $A_{f,F}^{\star,R}$ is less or equal to $m$.
\end{theorem}
\begin{proof}
In order to prove the theorem we need the main result of \cite{della} which states that given  $f\in\mathcal{C}^{2}(\Omega,\mathbb{R})$ and $G\in\mathcal{C}^{1}(\Omega,\mathbb{R}^{2m})$, where $\Omega$ is an open subset of $\mathbb{R}^{2m}$, then 
$$
A_{f,G}^{\star}:=\{\mathbf{x}\in\Omega\mid \nabla f(\mathbf{x})=G(\mathbf{x}),~ \Gamma(DG(\mathbf{x}))^{m}\neq 0\},
$$
if not empty, is a Borel set covered by a finite family $\{\Sigma_{(\alpha_{1},\dots,\alpha_{m})}\mid \alpha_{1},\dots,\alpha_{m}\in\mathbb{N},~1\leq \alpha_1 < \dots < \alpha_{m}\leq 2m\}$ of $m-$dimensional regularly embedded $\mathcal{C}^{1}$ submanifolds of $\mathbb{R}^{2m}$, and in particular, its Hausdorff dimension is less or equal to $m$. Recall from \cite{della} that denoting $\nabla f - G:=(\Phi_1,\dots,\Phi_{2m})$, then for each $(\alpha_{1},\dots,\alpha_{m})\in\mathbb{N}^{m}$ such that $1\leq \alpha_1 < \dots < \alpha_{m}\leq 2m$, 
$$
\Sigma_{(\alpha_{1},\dots,\alpha_{m})}:=\{\mathbf{x}\in\Omega\mid \Phi_{(\alpha_{1},\dots,\alpha_{m})}(\mathbf{x})=0,~ \operatorname{rank}D\Phi_{(\alpha_{1},\dots,\alpha_{m})}(\mathbf{x})=m \},
$$
where $\Phi_{(\alpha_{1},\dots,\alpha_{m})}:=(\Phi_{\alpha_1},\dots,\Phi_{\alpha_m})$.

Let us recall now from Proposition \ref{rlgrd} that $\nabla^{L}_{b}f=B^{\star}\nabla f$, $\nabla^{R}_{b}f=B \nabla f$. Consequently, the sets $A_{f,F}^{\star,L}, A_{f,F}^{\star,R}$ become
\begin{align*}
A_{f,F}^{\star,L}:&=\{\mathbf{x}\in\Omega\mid \nabla f(\mathbf{x})=(B^{\star})^{-1} F(\mathbf{x}),~ \Gamma((B^{\star})^{-1}DF(\mathbf{x}))^{m}\neq 0\},\\
A_{f,F}^{\star,R}:&=\{\mathbf{x}\in\Omega\mid \nabla f(\mathbf{x})=B^{-1} F(\mathbf{x}),~ \Gamma(B^{-1}DF(\mathbf{x}))^{m}\neq 0\},
\end{align*} 
Now the conclusion follows by applying the main result of \cite{della} to $G=(B^{\star})^{-1}F$ and $G=B^{-1}F$, respectively.
\end{proof}

Let us recall that if $n=2m\neq 0$, a natural class of geometric structures on $\mathbb{R}^{n}$ consists of the symplectic structures. As the main result of this note, Theorem \ref{MTHM}, takes place on $\mathbb{R}^{2m}$, let us point out its connection with the class of symplectic structures on $\mathbb{R}^{2m}$.

\begin{theorem}
Let $b$ be a symplectic form on $\mathbb{R}^{2m}$, and let $(b,B)$ be the associated geometric pair. Let $f\in\mathcal{C}^{2}(\Omega,\mathbb{R})$ and $F\in\mathcal{C}^{1}(\Omega,\mathbb{R}^{2m}_{b})$, where $\Omega$ is an open subset of $\mathbb{R}^{2m}_{b}$. Then the set
\begin{align*}
A_{f,F}^{\star}:=\{\mathbf{x}\in\Omega\mid X_f(\mathbf{x})=F(\mathbf{x}),~ \Gamma(B^{-1}DF(\mathbf{x}))^{m}\neq 0\},
\end{align*} 
if not empty, is a Borel set which admits a cover consisting of at most $\binom{2m}{m}$ regularly embedded $\mathcal{C}^{1}$ $m-$dimensional submanifolds of $\mathbb{R}^{2m}_{b}$. Particularly, the Hausdorff dimension of $A_{f,F}^{\star}$ is less or equal to $m$.
\end{theorem}
\begin{proof}
The proof follows directly from Theorem \ref{MTHM} taking into account that $b$ is skew--symmetric and hence $B^{\star}=-B$ and  $\nabla^{L}_{b}f = -\nabla^{R}_{b}f= X_f$, where $X_f$ is the Hamiltonian vector field of $f$ with respect to the symplectic form $b$.
\end{proof}

In the case when $b$ is an Euclidean or pseudo-Euclidean inner product on $\mathbb{R}^{n}$ (with $n$ a nonzero even number) the Theorem \ref{MTHM} becomes as follows (recall that in this case $\nabla^{L}_{b} = \nabla^{R}_{b}=:\nabla_{b}$).

\begin{theorem}\label{MT1}
Let $b$ be an Euclidean or pseudo-Euclidean inner product on $\mathbb{R}^{2m}$, and let $(b,B)$ be the associated geometric pair. Let $f\in\mathcal{C}^{2}(\Omega,\mathbb{R})$ and $F\in\mathcal{C}^{1}(\Omega,\mathbb{R}^{2m}_{b})$, where $\Omega$ is an open subset of $\mathbb{R}^{2m}_{b}$. Then the set
\begin{align*}
A_{f,F}^{\star}:=\{\mathbf{x}\in\Omega\mid \nabla_b f(\mathbf{x})=F(\mathbf{x}),~ \Gamma(B^{-1}DF(\mathbf{x}))^{m}\neq 0\},
\end{align*} 
if not empty, is a Borel set which admits a cover consisting of at most $\binom{2m}{m}$ regularly embedded $\mathcal{C}^{1}$ $m-$dimensional submanifolds of $\mathbb{R}^{2m}_{b}$. Particularly, the Hausdorff dimension of $A_{f,F}^{\star}$ is less or equal to $m$.
\end{theorem}
\begin{proof}
The proof follows directly from Theorem \ref{MTHM} taking into account that $b$ is symmetric and hence $B^{\star}=B$ and  $\nabla^{L}_{b}f = \nabla^{R}_{b}f= \nabla_{b}f$.
\end{proof}

\begin{remark}
In Theorem \ref{MT1} if $b$ is the canonical Euclidean inner product on $\mathbb{R}^{2m}$, we get $B=\operatorname{Id}_{\mathbb{R}^{2m}}$, $\nabla_b=\nabla$, and consequently we recover the main result of \cite{della}.
\end{remark}

%\subsection*{Acknowledgment}
%This work was supported by a grant of the Romanian National Authority for Scientific Research, CNCS-UEFISCDI, project number PN-II-RU-TE-2011-3-0103. 

\bigskip
\bigskip

\noindent {\sc R.M. Tudoran}\\
West University of Timi\c soara\\
Faculty of Mathematics and Computer Science\\
Department of Mathematics\\
Blvd. Vasile P\^arvan, No. 4\\
300223 - Timi\c soara, Rom\^ania.\\
E-mail: {\sf razvan.tudoran@e-uvt.ro}\\
\medskip

\end{document}